\documentclass[11pt,a4paper,oneside]{amsart}

\usepackage[dvipsnames]{xcolor}
\usepackage[pdftex,  colorlinks=true]{hyperref}
\hypersetup{urlcolor=RoyalBlue, linkcolor=RoyalBlue,  citecolor=black}
\usepackage{amssymb} 
\usepackage{amsthm} 
\usepackage[english]{babel} 
\usepackage[font=small,justification=centering]{caption} 
\usepackage[nodayofweek]{datetime}
\usepackage[T1]{fontenc}
\usepackage[a4paper]{geometry}
\usepackage[utf8]{inputenc}
\usepackage{ifthen}
\usepackage{mathtools}
\usepackage[dvipsnames]{xcolor}
\usepackage{adjustbox}
\usepackage{setspace}
\usepackage{tikz-cd}
\usepackage{tikz}
\usepackage{tikzit}
\usepackage{xfrac}
\usepackage{cleveref}
\usepackage{crossreftools}
\usepackage{stmaryrd} 
\usepackage{bm} 
\usepackage{comment} 
\usepackage{geometry}
\usepackage{enumitem} 
\usepackage{physics}
\usepackage{wasysym}
\usepackage{csquotes}
\usepackage{quoting,xparse}
\usepackage{relsize}
\usepackage[titletoc,toc,page]{appendix}

\definecolor{hotpink}{rgb}{255, 0, 136} 

\newtheorem{theorem}{Theorem}[section]
\newtheorem{proposition}[theorem]{Proposition}
\newtheorem{lemma}[theorem]{Lemma}
\newtheorem{remark}[theorem]{Remark}

\newtheorem{corollary}[theorem]{Corollary}

\newtheorem*{proposition*}{Proposition}
\newtheorem*{notation*}{Notation}
\newtheorem*{lemma*}{Lemma}
\newtheorem*{fact*}{Fact}
\newtheorem*{remark*}{Remark}
\newtheorem*{bigassumption*}{Big Assumption}
\newtheorem*{smallassumption*}{Small Assumption}
\newtheorem*{question*}{Question}
\newtheorem*{theorem*}{Theorem}
\newtheorem*{propprobtrue*}{A proposition that may or may not be true}
\newtheorem*{conjecture*}{Conjecture}
\newtheorem*{claim*}{Claim}
\newtheorem*{case1*}{Case 1}
\newtheorem*{case1a*}{Case 1a}
\newtheorem*{case1b*}{Case 1b}
\newtheorem*{case2*}{Case 2}
\newtheorem*{case3*}{Case 3}
\newtheorem*{claim1*}{Claim 1}
\newtheorem*{claim2*}{Claim 2}
\newtheorem*{claim3*}{Claim 3}
\newtheorem*{claim4*}{Claim 4}
\newtheorem*{claim5*}{Claim 5}
\newtheorem*{claim6*}{Claim 6}
\newtheorem*{claim7*}{Claim 7}
\newtheorem*{claim8*}{Claim 8}
\newtheorem*{corollary*}{Corollary}
\newtheorem*{armkon*}{A remark on the proofs of the two propositions above}

\theoremstyle{definition}

\newtheorem*{example*}{Example}
\newtheorem*{definition*}{Definition}

\newtheorem{definition}[theorem]{Definition}

\tikzstyle{black dot}=[fill=black, draw=black, shape=circle, scale=0.3]

\tikzstyle{dotted?}=[-, dotted, thick]
\tikzstyle{arrow}=[draw={rgb,255: red,255; green,0; blue,136}, {|->}, fill=none]
\tikzstyle{dashed?}=[-, dashed]
\tikzstyle{blueline}=[-, draw={rgb,255: red,28; green,116; blue,222}]
\tikzstyle{pinkline}=[-, draw={rgb,255: red,255; green,0; blue,136}]
\tikzstyle{blackarrow}=[->]
\tikzstyle{greenline}=[-, draw={rgb,255: red,55; green,255; blue,0}]

\def\T{{\mathcal T}}

\def\Y{{\mathbb Y}}

\def\T{{\mathcal T}}

\def\C{{\mathcal{C}}}

\def\pc{\mathcal{P}_K(\mathbb Y)}
\def\qtoms{\mathcal{C}_K(\mathbb Y)}
\def\toms{\mathcal{C}_K^\T(\mathbb Y)}

\def\diam{\textrm{diam}}

\makeatletter
\newcommand{\myitem}[1]{%
\item[#1]\protected@edef\@currentlabel{#1}%
}
\makeatother

\DeclarePairedDelimiter\absval{\lvert}{\rvert}

\title{Two new constructions in the theory of projection complexes}

\date{}

\author{Patrick S. Nairne}

\begin{document}

\emergencystretch 3em

\setcounter{tocdepth}{1}

\raggedbottom

\begin{abstract}
In this paper we provide two new constructions that are useful for the theory of projection complexes developed by Bestvina, Bromberg, Fujiwara and Sisto. We prove that there exists a subtree of the projection complex which is quasiisometric to the projection complex. We use this subtree to form a tree of metric spaces, which is a subgraph of the quasi-tree of metrics spaces and quasiisometric to it. These constructions simplify the metric structure (up to quasiisometry) of the projection complex and the quasi-tree of metric spaces. As an application, we use these constructions to provide a shorter proof of Hume's theorem that the mapping class group admits a quasiisometric embedding into a finite product of trees.
\end{abstract}

\maketitle

\section{Introduction}

\subsection{Results}

Various metric spaces can be understood by considering projections of subspaces onto other subspaces. For example, due to work of Sisto, the metric of a relatively hyperbolic group can be understood by considering projections of peripheral cosets onto other peripheral cosets \cite{SISTO}. Similarly, the metric on the mapping class group $MCG(S)$ can be understood by considering the projection of curves in $S$ onto curve complexes of subsurfaces of $S$ \cite{MM} \cite{MM2}. Bestvina, Bromberg and Fujiwara \cite{BBF1} innovated two structures that you can associate to such metric spaces: the \textit{projection complex} $\pc$ and the \textit{quasi-tree of metric spaces} $\qtoms$. Bestvina, Bromberg and Fujiwara also found a general set of conditions that a set of projections must satisfy such that $\pc$ and $\qtoms$ exist.

A key result in the theory of projection complexes states that $\pc$ is a quasi-tree \cite[Theorem 3.16]{BBF1}. However, in the work of Bestvina, Bromberg and Fujiwara, this result is proved abstractly via Manning's bottleneck condition \cite{MAN}. In this paper, I am able to prove this result more concretely. 

\begin{theorem} \label{intro.thm.subtree}
There exists subtree $\T \subseteq \pc$ such that the inclusion map $\T \hookrightarrow \pc$ is a quasiisometry.
\end{theorem}

\begin{proof}
See \Cref{thm.subtree1} and \Cref{thm.subtree2}.
\end{proof}

Furthermore this subtree $\T$ has a simple definition: having arbitrarily chosen a basepoint $B \in \Y$, the subgraph $\T$ is the union of all standard paths $\Y_K[B,X]$ in the projection complex. 

Analogously, in \cite{HUME}, Hume proves that the quasi-tree of metric spaces $\qtoms$ satisfies a \textit{relative bottleneck property}. Hume then shows that this implies that $\qtoms$ is quasiisometric to a tree-graded space, whose pieces are exactly the metric spaces $\C(Y)$ which form $\qtoms$. Again, this result can be proved more concretely.

\begin{theorem} \label{intro.thm.toms}
There exists a tree of metric spaces $\toms \subseteq \qtoms$ such that the inclusion map $\toms \hookrightarrow \qtoms$ is a quasiisometry.
\end{theorem}

\begin{proof}
See \Cref{thm.toms}.
\end{proof}

And again $\toms$ has a simple definition: $\toms$ is formed by only keeping the \textit{transverse} edges of $\qtoms$ which correspond to edges of $\T$.

As an application of \Cref{intro.thm.toms}, we reprove Hume's theorem \cite[Corollary 2]{HUME} that the mapping class group admits a quasiisometric embedding into a finite product of trees. 

\begin{theorem} \label{intro.thm.mcg}
Let $S_{g,n}$ be a compact surface. Then $MCG(S_{g,n})$ admits a quasiisometric embedding into a finite product of trees. 
\end{theorem}

\subsection{Application to relatively hyperbolic groups}

The constructions in this paper can also be used to quasiisometrically embed certain relatively hyperbolic groups into finite products of binary trees, as described in an upcoming preprint (or as found in previous versions of this preprint).

\subsection{Structure of this paper}

In \Cref{sec.pc}, the relevant background on projection complexes and the quasi-tree of metric spaces is provided. In \Cref{sec.subtree}, the subtree $\T$ is defined and \Cref{intro.thm.subtree} is proved. In \Cref{sec.toms}, the tree of metric spaces $\toms$ is defined and \Cref{intro.thm.toms} is proved. In \Cref{sec.app}, the new proof of Hume's theorem \Cref{intro.thm.mcg} is provided. 

\subsection{Acknowledgements}

Thank you to Alessandro Sisto for useful discussions on how to present this work. 

\tableofcontents

\section{Background} \label{sec.pc}

\subsection{Projection complexes}

The theory of projection complexes and quasi-trees of metric spaces is outlined thoroughly and properly in \cite{BBFS}, nonetheless in this section I briefly provide the necessary background. 

\begin{definition} \label{def.projections}
Let $\Y$ be an indexing set such that for each $Y \in \Y$ we have an associated geodesic metric space $\C(Y)$. Suppose also that for each $Y \in \Y$ we have a function
\[\pi_Y: \Y \setminus \{Y\} \rightarrow \textrm{non-empty subsets of } \C(Y)\]
We call these functions \textit{projections}. We then define $d_Y^\pi: \Y \setminus \{Y\} \times \Y \setminus \{Y\} \rightarrow [0,\infty]$ by
\[d_Y^\pi(X,Z) = \diam(\pi_Y(X) \cup \pi_Y(Z))\]
We call these functions the \textit{projection distances}. The collection of maps $\{\pi_Y\}$ satisfy the \textit{projection axioms} if there exists a constant $\xi < \infty$, known as the \textit{projection constant}, such that
\begin{enumerate}
    \myitem{(P3)}\label{condition.p3} $\diam(\pi_Y(X)) \leq \xi$ for all distinct $X,Y \in \Y$;
    \myitem{(P4)}\label{condition.p4} if $X,Y,Z \in \Y$ are distinct and such that $d_Y^\pi(X,Z) > \xi$ then $d_X^\pi(Y,Z) \leq \xi$;
    \myitem{(P5)}\label{condition.p5} $\{Y : d_Y^\pi(X,Z) > \xi\}$ is finite for each pair $X,Z \in \Y \setminus \{Y\}$.
\end{enumerate}

The axioms \ref{condition.p3} - \ref{condition.p5} are those that appear in "real-life". However we need them to satisfy an extra axiom in order to conduct the theory of projection complexes. The collection of maps $\{\pi_Y\}$ satisfy the \textit{strong projection axioms} if, in addition to \ref{condition.p3} - \ref{condition.p5}, the following condition also holds
\begin{enumerate}       
    \myitem{(P4')}\label{condition.p4'} if $X,Y,Z \in \Y$ are distinct and such that $d_Y^\pi(X,Z) > \xi$ then $\pi_X(Y) = \pi_X(Z)$. 
\end{enumerate}
\end{definition}

Bestvina, Bromberg and Fujiwara prove that, given projections $\{\pi_Y\}$ satisfying the projection axioms, we can adjust the projections $\pi_Y(X)$ by a uniformly bounded amount so that the projections $\{\pi_Y\}$ now satisfy the strong projection axioms with respect to some possibly larger projection constant $\theta$. This is the content of the following theorem (see \cite[Theorem 4.1]{BBFS}). 

\begin{theorem}[Bestvina--Bromberg--Fujiwara] \label{thm.modification}
Suppose we have a collection of maps $\{\pi_Y\}$ satisfying the projection axioms with projection constant $\xi$. Then there exists a new collection of projections 
\[ \{\pi_Y': \Y \setminus \{Y\} \rightarrow \textrm{non-empty subsets of } \C(Y) \}\]
such that if $d_Y(X,Z) := \diam(\pi_Y'(X) \cup \pi_Y'(Z))$ then
\begin{itemize}
    \item $\pi_Y'(X) \subseteq N_\xi(\pi_Y(X))$;
    \item $\{\pi_Y'\}$ now satisfy the strong projection axioms \ref{condition.p3} - \ref{condition.p5} and \ref{condition.p4'} with respect to the projection constant $\theta = 11\xi$;
    \item $d_Y^\pi - 2\xi \leq d_Y \leq d_Y^\pi + 2\xi$.
\end{itemize}
Further, if all the metric spaces $\C(Y)$ are graphs, and all the projections $\pi_Y(X) \subset \C(Y)$ are collections of vertices, then $\pi_Y'(X) \subset \C(Y)$ is also a collection of vertices. 
\end{theorem}

The \textit{Further} part of the above theorem may be deduced from the definition of $\pi_Y'(X)$ in \cite[Definition 4.11]{BBFS}. Indeed, one may write $\pi_Y'(X)$ as a union $\bigcup_W \pi_Y(W)$ where the union is over a non-empty set. 

Given a collection of projections $\{\pi_Y : Y \in \Y\}$ with projection distances $d_Y: \Y \setminus \{Y\} \times \Y \setminus \{Y\} \rightarrow [0,\infty)$ and projection constant $\theta$ satisfying the \textit{strong} projection axioms, we can now define standard paths, the projection complex and the quasi-tree of metric spaces. 

\begin{definition} \label{def.projectioncomplexsp}
For a constant $K \geq 0$, define 
\[\Y_K(X,Z) = \{Y \in \Y : d_Y(X,Z) > K\}\]
and
\[\Y_K[X,Z] = \Y_K(X,Z) \cup \{X,Z\}\]
In what follows, we will also use the notation
\[L(X,Z) = \absval{\Y_K(X,Z)}\]
and
\[L[X,Z] = \absval{\Y_K[X,Z]}\]
\end{definition}

One should interpret the statement $Y \in \Y_K(X,Z)$ as saying that $Y$ is \textit{between} $X$ and $Z$. 

\begin{definition} \label{def.projectioncomplex}
The \textit{projection complex} is the graph $\pc$ with vertex set $\Y$ and an edge between two vertices if $\Y_K(X,Z) = \emptyset$.
\end{definition}

We now state three results which describe the structure of the sets $\Y_K[X,Z]$. These are Lemma 2.3, Proposition 2.4 and Corollary 2.6 of \cite{BBFS} respectively.

\begin{lemma}[Bestvina--Bromberg--Fujiwara] \label{lem.projections}
For $K \geq 2\theta$, and $Y,Y' \in \Y_K(X,Z)$, the following conditions are equivalent
\begin{enumerate}
    \item $d_Y(X,Y') > \theta$;
    \item $d_{Y'}(Y,Z) > \theta$;
    \item $d_Y(Y',Z) \leq \theta$;
    \item $d_{Y'}(X,Y) \leq \theta$;
    \item $d_Y(Y',W) = d_Y(Z,W)$ for all $W \neq Y$;
    \item $d_{Y'}(Y,W) = d_{Y'}(X,W)$ for all $W \neq Y'$.
\end{enumerate}
\end{lemma}

Given $K \geq 2\theta$ and two distinct elements $Y, Y' \in \Y_K(X,Z)$, we write $Y < Y'$ if any of the above equivalent conditions hold. Write $X < Y$ and $Y < Z$ for all $Y \in \Y_K(X,Z)$ and write $X < Z$.

\begin{proposition}[Bestvina--Bromberg--Fujiwara] \label{prop.projections}
Suppose $K \geq 2\theta$. The relation $<$ on $\Y_K[X,Z]$ defines a total order with least element $X$ and greatest element $Z$. Further, if $Y < Y' < Y''$ in $\Y_K[X,Z]$ then $d_{Y'}(Y, Y'') = d_{Y'}(X,Z)$.
\end{proposition}

\begin{corollary}[Bestvina--Bromberg--Fujiwara] \label{cor.projections}
Let $K \geq 2\theta$. Suppose $Y, Y'' \in \Y_K[X,Z]$ and for some $Y' \in \Y$ we have $d_{Y'}(Y,Y'') > K$. Then $Y' \in \Y_K(X,Z)$.
\end{corollary}

We now state some important facts about the graph $\pc$. These are Lemma 3.1, Theorem 3.5, Lemma 3.6 and Corollary 3.7 from \cite{BBFS} respectively.

\begin{lemma}[Bestvina--Bromberg--Fujiwara] \label{lem.projections2}
Let $K \geq 2\theta$. Suppose $\Y_K[X,Z] = \{X < X_1 < \dots < X_k < Z\}$. Then
\[X \rightarrow X_1 \rightarrow \dots \rightarrow X_k \rightarrow Z\] 
is a path in $\pc$. 
\end{lemma}

We refer to the sets $\Y_K[X,Z]$ as \textit{standard paths} in $\pc$. Note that for any $X,Z \in \Y$ we have $\Y_K[X,Z] = \Y_K[Z,X]$ as a set, and $\Y_K[Z,X]$ is the path $\Y_K[X,Z]$ walked in reverse.  

\begin{theorem}[Bestvina--Bromberg--Fujiwara] \label{thm.pcquasi-tree}
For $K \geq 3\theta$, $\pc$ is a quasi-tree.  
\end{theorem}

\begin{lemma}[Bestvina--Bromberg--Fujiwara] \label{lem.projections3}
Suppose $K \geq 2\theta$. Let $X,Y,Z$ be distinct elements of $\Y$. Write $\Y_K[X,Y] = \{X = Y_0 < Y_1 < \dots < Y_k < Y_{k+1} < Y\}$ and $\Y_K[Y,Z] = \{Y = Z_0 < Z_1 < \dots < Z_l < Z_{l+1} < Z\}$. Then the standard path $\Y_K[X,Z]$ is the concatenation of three disjoint segments:
\begin{itemize}
    \item it begins with an initial segment $X \rightarrow Y_1 \rightarrow \dots \rightarrow Y_i$ in $\Y_K[X,Y]$ for some $0 \leq i \leq k$;
    \item it then has a (possibly empty) middle segment with at most $2$ vertices which are in neither $\Y_K[X,Y]$ or $\Y_K[Y,Z]$;
    \item finally, it ends with a terminal segment $Z_{j} \rightarrow Z_{j+1} \rightarrow \dots \rightarrow Z_l \rightarrow Z$ for some $1 \leq j \leq l+1$.
\end{itemize}
\end{lemma}

The above lemma is the key to understanding the geometry of the projection complex. It says that the standard paths $\Y_K[X,Z]$, $\Y_K[X,Y]$, $\Y_K[Y,Z]$ \textit{almost} form a tripod. Instead, the three standard paths have a cycle of length at most $9$ in their centre. See \Cref{fig.pctriangle}.

\begin{figure}[h]
\centering
\includegraphics[width=0.5\textwidth]{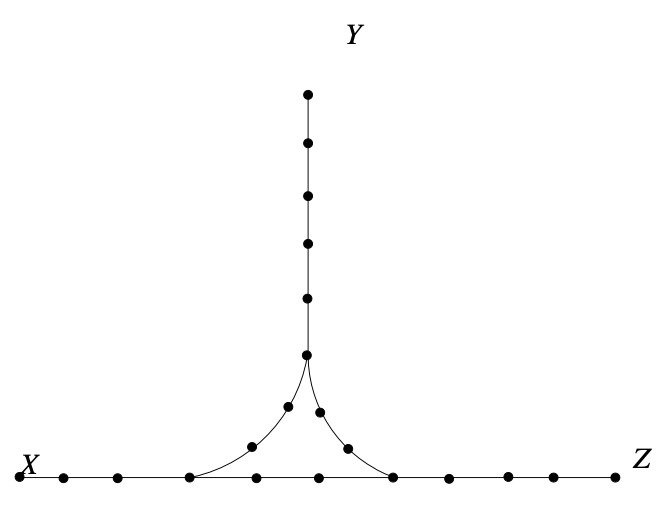} 
\caption{A triangle of standard paths. Taken from \cite{BBFS}.}
\label{fig.pctriangle}
\end{figure}

\begin{corollary}[Bestvina--Bromberg--Fujiwara] \label{cor.pcdist}
For $K \geq 3\theta$, standard paths are quasigeodesics; for all $X,Z \in \Y$ we have
\[(L[X,Z] - 1)/2 \leq d_{\pc}(X,Z) \leq L[X,Z] - 1\]
\end{corollary}

\subsection{Quasi-trees of metric spaces}

Suppose again that $\{\pi_Y : Y \in \Y\}$ satisfies \ref{condition.p3} - \ref{condition.p5} and  \ref{condition.p4'} with respect to the projection constant $\theta$. Suppose also that $K \geq 3\theta$. 

\begin{definition} \label{def.qtoms}
The \textit{quasi-tree of metric spaces} $\qtoms$ is built by taking the disjoint union of the metric spaces $\C(Y)$ and then adding an edge of length $K$ between every element of $\pi_X(Y)$ and every element of $\pi_Y(X)$ whenever there is an edge between $X$ and $Y$ in $\pc$. We will often refer to one of the edges between $\pi_X(Y)$ and $\pi_Y(X)$ as a \textit{transverse edge}. 
\end{definition}

The name of $\qtoms$ comes from the fact that crushing all the disjoint metric spaces $\C(Y)$ that make up $\qtoms$ to points leaves us with the quasi-tree $\pc$. 

It is useful to extend the definitions of $\pi_Y$, $d_Y$ and $\Y_K(\cdot , \cdot )$ over to elements $x \in \C(X) \subseteq \qtoms$.

\begin{definition} \label{def.projofpoints}
For $x \in \C(X) \subseteq \qtoms$ and $z \in \C(Z) \subseteq \qtoms$ and $Y \in \Y$ 
\begin{itemize}
    \item we define 
    $\pi_Y(x) = 
    \begin{cases}
    \pi_Y(X) & \text{if $Y \neq X$} \\
    \{x\} & \textrm{if $Y = X$}
    \end{cases}
    $;
\item we define $d_Y(x,z) = \diam(\pi_Y(x) \cup \pi_Y(z))$ where the diameter is taken with respect to the metric on $\C(Y)$;
\item we define $\Y_K(x,z) = \{Y \in \Y : d_Y(x,z) > K\}$.
\end{itemize}
\end{definition}

We then have the following distance formula for $\qtoms$ \cite[Theorem 6.4]{BBFS}. It says that we can approximate distances in $\qtoms$ by just looking at the projections.

\begin{theorem}[Bestvina--Bromberg--Fujiwara] \label{thm.qtomsdistanceformula}
If $K \geq 4\theta$, then for all $x \in \C(X) \subseteq \qtoms$ and $z \in \C(Z) \subseteq \qtoms$ we have
\[\frac{1}{4} \sum_{Y \in \Y_K(x,z)} d_Y(x,z) \leq d_{\qtoms}(x,z) \leq 2 \sum_{Y \in \Y_K(x,z)} d_Y(x,z) + 3K\]
\end{theorem}

\section{A quasiisometric subtree of the projection complex} \label{sec.subtree}

Suppose we have a collection of projections $\{\pi_Y : Y \in \Y\}$ with projection distances $d_Y: \Y \setminus \{Y\} \times \Y \setminus \{Y\} \rightarrow [0,\infty)$ and projection constant $\theta$ that satisfy axioms \ref{condition.p3} - \ref{condition.p5} and \ref{condition.p4'}. Let $K \geq 4\theta$. 

\begin{definition} \label{def.tree}
Let $B \in \Y$ be a fixed basepoint. This corresponds to a vertex of $\pc$. We can define the following subgraph of $\pc$.
\[\T_K(\Y,B) := \bigcup_{X \in \Y} \Y_K[B,X] \]
We will generally write $\T = \T_K(\Y,B)$ for simplicity. In the above, $\Y_K[B,X]$ should be interpreted as a path that includes the relevant edges as described in \Cref{lem.projections2}.
\end{definition}

By definition $\T$ is a connected subgraph of $\pc$ that must contain all the vertices of $\pc$. 

\begin{theorem} \label{thm.subtree1}
The subgraph $\T \subseteq \pc$ is a tree.
\end{theorem}

First we need a preliminary lemma.

\begin{lemma} \label{lem.pcsp}
Let $\Y_K[X,Z] = \{X = X_0 < X_1 < \dots< X_k < X_{k+1} = Z\}$ be the standard path from $X$ to $Z$. Let $0 \leq i < j \leq k + 1$. Then $\{X_i < X_{i+1} < \dots < X_j\}$ is the standard path from $X_i$ to $X_j$. 
\end{lemma}

\begin{proof}
This follows from elementary arguments that use \Cref{lem.projections}, \Cref{prop.projections} and \Cref{cor.projections}.
\end{proof}

\begin{proof}[Proof of \Cref{thm.subtree1}]

We begin by proving the following claim.

\begin{claim*}
Suppose we have a standard path $\Y_K[B,X] = \{B = X_0 < X_1 < \dots< X_k < X_{k+1} = X\}$ and suppose there exists another vertex $Z$ connected by an edge (in $\T$) to $X$ such that $Z \neq X_k$. Then $\{B < X_1 < \dots < X_k < X < Z\}$ is the standard path from $B$ to $Z$. 
\end{claim*}

\begin{proof}[Proof of claim]
Since the edge from $X$ to $Z$ is in $\T$, we know that there exists some vertex $U \in \Y$ such that $\Y_K[B,U] = \{B = X_0' < X_1' < \dots < X_l' < X_{l+1}' = U\}$ and $\{X,Z\} = \{X_j', X_{j+1}'\}$. Suppose for a contradiction that $X = X_{j+1}'$ and $Z = X_j'$. Then, by \Cref{lem.pcsp}, $B < X_1' < X_2' < \dots < X_{j-1}' < Z < X$ is the standard path from $B$ to $X$. It follows that $Z = X_k$ which contradicts our assumption on $Z$. So $X = X_j'$ and $Z = X_{j+1}'$. By \Cref{lem.pcsp}, we know that $B < X_1' < \dots < X_j'$ is the standard path from $B$ to $X$. \Cref{lem.pcsp} also implies that $B < X_1' < \dots < X_{j+1}'$ is the standard path from $B$ to $Z$. Hence $B < X_1 < \dots < X_k < X < Z$ is the standard path from $B$ to $Z$. 
\end{proof}

Suppose in $\T$ we have a cycle $C = \{Y_0, Y_1, \dots, Y_{l-1}\}$ where there is an edge between $Y_i$ and $Y_{i+1}$ (taken modulo $l$). Let $Y \in C$ and suppose that the path $\Y_K[B,Y]$ first intersects $C$ at the vertex $Y' \in C$. It follows that the standard path $\Y_K[B,Y'] = \{B < X_1 < \dots < X_k < Y'\}$ only intersects $C$ at $Y'$. Suppose, for simplicity, that $Y' = Y_0$. By the claim, we know that $B < X_1 < \dots < X_k < Y_0 < Y_1$ is the standard path from $B$ to $Y_1$. Similarly, we know that $B < X_1 < \dots < X_k < Y_0 < Y_1 < Y_2$ is the standard path from $B$ to $Y_2$. Continuing in this way, we find that $B < X_1 < \dots < X_k < Y_0 < Y_1 < Y_2 < \dots < Y_l < Y_0$ is the standard path from $B$ to $Y_0$ which is of course nonsense. 
\end{proof}

\begin{theorem} \label{thm.subtree2}
The inclusion map $\T \hookrightarrow \pc$ is a quasiisometry. 
\end{theorem}

\begin{proof}
Evidently the inclusion map is Lipschitz; it remains to prove the lower bound of the quasiisometric inequality. 

Consider the triple $B,X,Z$. By \Cref{lem.projections3}, the path $\Y_K[X,Z]$ can be described as follows: an initial segment that is also initial within $\Y_K[X,B]$, a middle segment of length at most $2$, and a terminal segment that is also terminal in $\Y_K[B,Z]$. Let $Y'$ be the vertex at which $\Y_K[X,Z]$ diverges from $\Y_K[X,B]$ and let $Y''$ be the vertex at which $\Y_K[X,Z]$ joins $\Y_K[B,Z]$. Let $Y$ be the vertex at which $\Y_K[B,X]$ and $\Y_K[B,Z]$ diverge. By \Cref{lem.projections3} we have
\begin{equation}
L[X,Y] \leq L[X,Y'] + 3 \label{eqn.ub1}
\end{equation}
and
\begin{equation}
L[Y,Z] \leq L[Y'',Z] + 3 \label{eqn.ub2}
\end{equation}
Now, by the definition of $\T$, we have that
\begin{equation}
d_\T(X,Z) = d_\T(X,Y) + d_\T(Y,Z) = L[X,Y] - 1 + L[Y,Z] - 1 \label{eqn.treedist}
\end{equation}
We also know that 
\begin{equation}
L[X,Z] \geq L[X,Y'] + L[Y'', Z] - 1 \label{eqn.lb1}
\end{equation}
Hence 
\begin{align*}
d_{\pc}(X,Z) &\geq \frac{1}{2}(L[X,Z] - 1) & \textrm{by \Cref{cor.pcdist}} \\
&\geq \frac{1}{2}(L[X,Y'] + L[Y'', Z] - 1) & \textrm{by \eqref{eqn.lb1}} \\
&\geq \frac{1}{2}(L[X,Y] - 3 + L[Y, Z] - 3 - 1) & \textrm{by \eqref{eqn.ub1} and \eqref{eqn.ub2}} \\
&= \frac{1}{2}(d_\T(X,Z) - 5) &\textrm{by \eqref{eqn.treedist}}
\end{align*}
and so we are done.
\end{proof}

\section{The tree of metric spaces} \label{sec.toms}

We can use $\T = \T(\Y,B)$ to define a \textit{tree} of metric spaces $\toms$ as follows.

\begin{definition} \label{def.toms}
For each pair $X,Y \in \Y$ which are connected by an edge in $\pc$, fix an arbitrary element $p(X,Y) \in \pi_X(Y) \subset \C(X)$ and an arbitrary element $p(Y,X) \in \pi_Y(X) \subset \C(Y)$. In order to define $\toms$, we take the disjoint union of the metric spaces $\C(X)$ and draw an edge of length $K$ between $p(X,Y)$ and $p(Y,X)$ whenever there is an edge between $X$ and $Y$ in $\T$.
\end{definition}

\begin{remark}
All sensible logic suggests that we should be using the notation $p_X(Y)$ instead of $p(X,Y)$ in order to properly mimic the existing notation $\pi_X(Y)$. However when the elements of $\Y$ have a subscript in their notation, or multiple subscripts, then the notation $p_{X}(Y)$ becomes painful to read. 
\end{remark}

Evidently, there exists a natural inclusion map $\toms \hookrightarrow \qtoms$. 

\begin{theorem} \label{thm.toms}
The inclusion $\toms \hookrightarrow \qtoms$ is a quasiisometry.
\end{theorem}

We first need a lemma about the projection complex. 

\begin{lemma} \label{lem.lemfortomsthm}
Let $X,Y,Z$ be distinct vertices of $\pc$. 
\begin{enumerate}
    \item If $W$ is an element of $\Y_K[X,Y]$ and $\Y_K[X,Z]$ but not $\Y_K[Y,Z]$ then $d_W(Y,Z) \leq K$.
    \item If $W$ is an element of $\Y_K[X,Y]$ but not $\Y_K[X,Z]$ and not $\Y_K[Y,Z]$ then $d_W(X,Y) \leq 2K$.
\end{enumerate}
\end{lemma}
\begin{proof}
The first statement is immediate. For the second, note that $d_W(X,Y) \leq d_W(X,Z) + d_W(Z,Y) \leq 2K$.
\end{proof}

\begin{proof}[Proof of \Cref{thm.toms}]

Coarse surjectivity is almost immediate: every element $x \in \qtoms$ is at distance at most $K/2$ from a metric space $\C(Y) \subset \toms$. Further, it is clear that the inclusion map is Lipschitz. It remains to prove the lower bound of the quasiisometric inequality. 

Let $x \in \C(X)$ and $z \in \C(Z)$. To produce a lower bound on $d_{\qtoms}(x,z)$, it is sufficient to find a nice path from $x$ to $z$ in $\toms$ whose length is comparable to the distance from $x$ to $z$ in $\qtoms$.

We begin by finding a quasigeodesic from $x$ to $z$ in $\qtoms$ (we will then find a path from $x$ to $z$ in $\toms$ of a comparable length). Suppose that $\Y_K[X,Z] = \{X = X_0 < X_1 < \dots < X_{n-1} < X_{n} = Z\}$. We have the following path $\gamma$ from $x$ to $z$ in $\qtoms$

\[
\adjustbox{scale=1,center}{
\begin{tikzcd}[cramped]
	x & {p(X_0,X_1)} \\
	& {p(X_1,X_0)} & {p(X_1,X_2)} \\
	&& {p(X_2,X_1)} & \dots & {p(X_{n-2},X_{n-1})} \\
	&&&& {p(X_{n-1},X_{n-2})} & {p(X_{n-1},X_{n})} \\
	&&&&& {p(X_{n},X_{n-1})} & z
	\arrow[from=1-1, to=1-2]
	\arrow[from=1-2, to=2-2]
	\arrow[from=2-2, to=2-3]
	\arrow[from=2-3, to=3-3]
	\arrow[from=3-3, to=3-4]
	\arrow[from=3-4, to=3-5]
	\arrow[from=3-5, to=4-5]
	\arrow[from=4-5, to=4-6]
	\arrow[from=4-6, to=5-6]
	\arrow[from=5-6, to=5-7]
\end{tikzcd}
}
\]
where each arrow either corresponds to a geodesic in $\C(X_i)$ or to a transverse edge between pairs of disjoint metric spaces. A segment $p(X_i,X_{i+1}) \rightarrow p(X_{i+1},X_{i})$ has length $K$. Each segment $p(X_{i},X_{i-1}) \rightarrow p(X_{i},X_{i+1})$ has length at most $d_{X_{i}}(x,z)$. Similarly, $x \rightarrow p(X_0,X_1)$ has length at most $d_{X_0}(x,z)$ and $p(X_n,X_{n-1}) \rightarrow z$ has length at most $d_{X_n}(x,z)$. Thus the length of $\gamma$ is at most 
\[\sum_{i} d_{X_i}(x,z) + nK \]
It is shown in the proof of \cite[Theorem 6.3]{BBFS} that
\[\sum_{i} d_{X_i}(x,z) + nK \leq 2 \sum_{Y \in \Y_K(x,z)} d_Y(x,z) + 3K\]
Applying \Cref{thm.qtomsdistanceformula} we see that
\[2 \sum_{Y \in \Y_K(x,z)} d_Y(x,z) + 3K \leq 8 d_{\qtoms}(x,z) + 3K \]
and so $\gamma$ has length at most $8 d_{\qtoms}(x,z) + 3K$. So $\gamma$ is an $(8,3K)$-quasigeodesic in $\qtoms$. 

We will now find a path in $\toms$ that is very close to $\gamma$. Let $X_i$ be the vertex of $\pc$ at which $\Y_K[X,Z]$ and $\Y_K[X,B]$ diverge (if they never diverge, then we are done, as $\gamma$ will lie inside $\toms$). Let $X_j$ be the vertex of $\pc$ at which $\Y_K[Z,X]$ and $\Y_K[Z,B]$ diverge (if they never diverge, then we are done, as $\gamma$ will lie inside $\toms$). Finally, let $Y$ be the vertex of $\pc$ at which $\Y_K[B,X]$ and $\Y_K[B,Z]$ diverge (if they never diverge, then we are done, as $\gamma$ will lie inside $\toms$). We can also assume that $X_i$, $X_j$ and $Y$ are distinct: if any of them are equal then $\gamma$ lies inside $\toms$. By \Cref{lem.projections3}, we have a geodesic of length at most $6$ from $X_i$ to $X_j$ in $\T$. This geodesic in $\T$ has the form $X_i \rightarrow \dots \rightarrow Y \rightarrow \dots \rightarrow X_j$. Now let $\eta$ be a geodesic in $\toms$ from $p(X_i,X_{i+1})$ to $p(X_j,X_{j-1})$.
The path $\eta$ is a "diversion" from $\gamma$ between $p(X_i,X_{i+1})$ and $p(X_j,X_{j-1})$. If we can show that $\eta$ has uniformly bounded length then we will be done, as the path which travels along $\gamma$ but takes the diversion $\eta$ will be a path in $\toms$ of comparable length to $\gamma$. 

The only metric spaces $\C(W)$ that $\eta$ travels through are those corresponding to the geodesic $X_i \rightarrow \dots \rightarrow Y \rightarrow \dots \rightarrow X_j$ in $\T$ described previously. The intersection of $\eta$ with $\C(X_i)$ corresponds to case (1) of \Cref{lem.lemfortomsthm} and hence has length at most $K$. Similarly, the intersection of $\eta$ with $\C(Y)$ or $\C(X_j)$ has length $K$. The intersection of $\eta$ with a metric space other than $\C(X_i), \C(Y), \C(X_j)$ corresponds to case (2) of \Cref{lem.lemfortomsthm} and hence has length at most $2K$. We conclude, by also taking into account the length of the transverse edges that $\eta$ travels along, that $\eta$ has length at most $17K$. \end{proof}

\section{A new proof of Hume's theorem} \label{sec.app}

\begin{theorem} \label{thm.mcg}
Let $S_{g,n}$ be a compact surface. Then $MCG(S_{g,n})$ admits a quasiisometric embedding into a finite product of trees. 
\end{theorem}

\begin{proof}
Bestvina, Bromberg and Fujiwara \cite[Theorem C]{BBF1} prove that there exists a quasiisometric embedding of the mapping class group into a finite product of quasiitrees of metric spaces
\[MCG(S_{g,n}) \rightarrow \prod_{i=1}^k \C_K(\Y^i) \]
where the metric spaces $\C(Y) \subset \C_K(\Y^i)$ are curve complexes of subsurfaces of $S_{g,n}$. By \Cref{intro.thm.toms}, we know that $\C_K(\Y^i)$ is quasiisometric to $\C_K^\T(\Y^i)$ which is a \textit{tree of curve complexes}. It can be proved using the results of \cite{MM} \cite{BF} \cite{MS} \cite{BUYALO} that the curve complexes $\C(Y)$ admit a quasiisometric embedding into a finite product of trees (see \cite[Proposition 5.2]{HUME} and \cite[Corollary 5.3]{HUME} for how to chain these results together). Since there are only finitely many isometry-types of subsurface curve complexes $\C(Y)$ associated to $S_{g,n}$ we conclude that there exists a uniform constant $k$ such that each $\C(Y)$ quasiisometrically embeds into a product of $k$ trees. Elementary arguments then prove that $\C_K^\T(\Y^i)$, which is a \textit{tree of curve complexes}, quasiisometrically embeds into a finite product of trees. 
\end{proof}

\bibliographystyle{alpha}
\bibliography{ref.bib}

\end{document}